\documentclass[12pt,a4paper,reqno]{amsart}

\usepackage[utf8]{inputenc}
\usepackage{amsthm,amssymb,amsfonts}
\usepackage{latexsym}
\usepackage{amsmath}
\usepackage{mathrsfs}
\usepackage{calc}
\usepackage{color}

\newcommand{\norm}[1]{\Vert#1\Vert }

\newtheorem{theorem}{Theorem}
\newtheorem{lemma}[theorem]{Lemma}
\newtheorem{definition}[theorem]{Definition}
\newtheorem{corollary}[theorem]{Corollary}
\newtheorem{proposition}[theorem]{Proposition}
\newtheorem{remark}[theorem]{Remark}

\title[Gabor Systems and Almost Periodic Functions]{Gabor Systems and Almost Periodic Functions}
\author[P. Boggiatto, C. Fern\'andez, A. Galbis]{Paolo Boggiatto, Carmen Fern\'{a}ndez and  Antonio Galbis}

\address{P. Boggiatto: Department of Mathematics, University of Torino, Via Carlo Alberto 10, 10123 Torino, Italy}
\email{paolo.boggiatto@unito.it}
\address{C. Fern\'{a}ndez: Departamento de An\'{a}lisis Matem\'{a}tico, Universidad de Valencia, Doctor Moliner 50,
46100 Burjasot (Valencia), Spain} \email{Carmen.Fdez-Rosell@uv.es}
\address{A. Galbis: Departamento de An\'{a}lisis Matem\'{a}tico, Universidad de Valencia, Doctor Moliner 50,
46100 Burjasot (Valencia), Spain} \email{Antonio.Galbis@uv.es}

\thanks{The research of C. Fern\'andez and A. Galbis was partially supported by the projects  MTM2013-43450-P and GVA Prometeo II/2013/013 (Spain).}

\date{}
 \dedicatory{}

\begin{document}

\maketitle

\begin{abstract}
We give a construction of Gabor type frames for suitable separable
subspaces of the non-separable Hilbert spaces $AP_2({\mathbb R})$ of almost
periodic functions of one variable. Furthermore we determine a
non-countable generalized frame for the whole space $AP_2({\mathbb R}).$ We show
furthermore that Bessel-type estimates hold for the $AP$ norm with
respect to a countable Gabor system using suitable almost periodic
norms of sequencies.
\end{abstract}

\section{Introduction}

Almost periodic functions on $\mathbb R$ are a natural
generalization of usual periodic functions. Their first definition
originates with the works of H. Bohr in 1924-1926. Since then a rich
theory has been developed by a number of authors, particularly
significant among others were the contributions by H. Weyl, N.
Wiener, A. S. Besicovich, S. Bochner, V. V. Stepanov, C. de la
Vall\'ee Puoissin and on LCA groups by J. Von Neumann.

Almost periodic functions found applications in various areas of
harmonic analysis, number theory, group theory. In connection with
dynamical systems and differential equations they were studied by B.
M. Levitan and V. V. Zhikov \cite{LeviZhik82} starting from results
of J. Favard. In the more general context of pseudo-differential
calculus M. A. Shubin \cite{Shub74} introduced scales of Sobolev
spaces of almost periodic functions and recently A. Oliaro, L.
Rodino and P.Wahlberg
 \cite{OliaDodiWahl12} extended some of these results to almost periodic Gevrey classes.

Another recent direction of research is connected with frames
theory. As the main spaces of almost periodic functions are
non-separable they cannot admit countable frames. The problem then
arises in which sense frame-type inequalities are still possible for
norm estimations in these spaces. Results in this direction were
obtained by F. Galindo \cite{Gali04}, and J.R. Partington - B.
Unalmis \cite{PartUnal01}, \cite{Unal13}, using windowed Fourier
transform and wavelet transforms, involving therefore Gabor and
wavelets type frames. Further developments were recently obtained by
Y.H. Kim and A. Ron \cite{KimRon09}, \cite{Kim13}, with a
fiberization techniques developed by A. Ron and Z. Shen
\cite{RonShen95}, \cite{RonShen05} in the context of shift-invariant
sysyems.

This paper has been inspired by the above-mentioned results of
Kim-Ron. With different techniques we shall recover and extend some
of their results concerning almost periodic norm estimates using
Gabor systems. We postpone the description after we have introduced
the definitions and properties we need. We refer e.g. to
\cite{Besi54}, \cite{Cord09}, \cite{CordGheoBarb89} for detailed
presentations of the theory of almost periodic functions, and e.g.
to \cite{Chri03}, \cite{Groc01}, \cite{Heil11} for frames theory.

The first type of functions we consider is, according to Bohr, the
following.
\begin{definition}\label{Bohr}
{\rm The space $AP({\mathbb R})$ of {\sl almost periodic functions} is the set of
continuous functions $f:\mathbb R\rightarrow \mathbb C$ with the
property that for every $\epsilon>0$, there exists $L>0$, such that
every interval of the real line of length greater than $L$ contains a
value $\tau$ satisfying
$$
\sup_t |f(t+\tau)-f(t)|\le \epsilon.
$$
}
\end{definition}
We recall that $AP({\mathbb R})$ coincides with the uniform norm closure of the
space of trigonometric polynomials $$\begin{displaystyle}\sum_{j=1}^Nc_j
e_{\lambda_j}\end{displaystyle}$$ with $e_{\lambda_j}=e^{i\lambda_j t}$ and
$\lambda_j\in \mathbb R, \ c_j\in\mathbb C$.

The norm $\|.\|_{AP}$ associated with the inner product
$$
(f,g)_{AP}=\lim_{T\to+\infty}(2T)^{-1}\int_{-T}^{T}f(t)\overline{g(t)}\,
dt$$ makes $AP({\mathbb R})$ a non-complete, non-separable space.

\par\medskip
\begin{definition}{\rm
The completion of $AP({\mathbb R})$ with respect to this norm is the Hilbert
space $AP_2({\mathbb R})$ (sometimes also indicated as $B^2$) of {\sl Besicovich
almost periodic functions}. }
\end{definition}

$AP_2({\mathbb R})$ also coincides with the completion of the trigonometric
polynomials with respect to the norm $\|.\|_{AP}$. The set
$\{e_\lambda: \lambda\in\mathbb R\}$ is a non-countable orthonornal
basis of $AP_2({\mathbb R})$, consequently for every $f\in AP_2({\mathbb R})$ the expansion
$f=\sum_{\lambda}(f,e_\lambda)_{AP}\ e_\lambda$ holds, where at most
countably many terms are different form zero and the norm is given
by Parseval's equality $\|f\|_{AP}^2=\sum_{\lambda\in\mathbb
R}|(f,e_\lambda)_{AP}|^2$.
\par\medskip
We will refer to $(f,e_\lambda)_{AP}$ as the $\lambda$-th Fourier coefficient of $f\in AP({\mathbb R})$ and from now on we will write $\widehat{f} (\lambda) = (f,e_\lambda)_{AP}.$ For a given function $f\in L^1({\mathbb R})$ we also put
$$
\widehat{f}(\lambda):=\int_{{\mathbb R}}f(t)e^{-i\lambda t}\ dt.
$$ Since $AP({\mathbb R})\cap L^1({\mathbb R}) = \left\{0\right\}$ there is no possible confusion between the Fourier coefficients of an almost periodic function and the Fourier transform of an integrable function, although we use the same notation for both objects.
\par\medskip
As usual the operators of translation and modulation on $L^1({\mathbb R})$ are defined as
$$
T_x\psi(t) = \psi(t-x),\ \ M_\omega\psi(t) = e^{i\omega t}\psi(t).
$$
\par\medskip
Another important concept is the one of   almost periodic functions in the
sense of Stepanov defined as follows.

\par\medskip
\begin{definition}{\rm
The {\it Stepanov} class $\rm S^2$ is the closure of the trigonometric
polynomials in the amalgam space $W(L^2,L^\infty)$ of those measurable functions
$f$ such that
$$
\norm{f}_{W(L^2,L^\infty)} = \sup_{t\in {\mathbb
R}}\left(\int_t^{t+1}\left|f(s)\right|^2\ ds\right)^{\frac{1}{2}} < \infty.
$$
}
\end{definition}
See e.g. \cite{Groc01}, \cite{Heil11} for a presentation of Wiener
amalgam spaces. Between the above-mentioned spaces of almost
periodic functions the following continuous inclusions hold: $(AP({\mathbb R}),
\|.\|_{\infty})\subset (S^2,\|.\|_{W(L^2,L^\infty)}) \subset
(AP_2({\mathbb R}),\|.\|_{AP})$.

Finally an important role in our paper will be played by almost periodic sequences according to the following definition.
\begin{definition}
{\rm
A real or complex sequence $(a_j)_{j\in\mathbb N}$ is said {\sl almost periodic} if for every $\epsilon>0$ there exists $N\in\mathbb N$ such that any set of $N$ consecutive integers contains an integer $p$ satisfying
$$
\sup_{n\in\mathbb N} |a_{n+p}-a_n|\le \epsilon.
$$
}
\end{definition}
The definition is a discrete version of that of $AP({\mathbb R})$ and actually
almost periodic sequences can be characterized as restrictions to
the integers of functions in $AP({\mathbb R})$. The space of almost periodic
sequences, which is denoted by $AP(\mathbb Z)$, will be endowed with
the inner product
$$
(a,b)_{AP(\mathbb Z)}=\lim_{p\to+\infty}(2p)^{-1}\sum_{n=-p}^pa_n\overline{b_n}
$$
and the corresponding norm. For $\lambda\in [0, 2\pi)$, the
sequences $\tilde e_\lambda= (e^{i\lambda n})_{n\in\mathbb Z}$ are a
non-countable orthonormal system in $AP(\mathbb Z)$.
\par\medskip

The paper is organized in the following way. In section 2 we show how to construct an  almost periodic function from an almost periodic sequence in terms of translates of a fixed continuous function in the Wiener class. This can be viewed as an extension of the classical periodization lemma. Then, we characterize, for a given $\psi \in L^1(\mathbb{R}),$ the continuity of the map $$S:\left(AP(\mathbb{R}),\norm{\cdot}_{AP}\right)\to \left(AP(\mathbb{Z}),\norm{\cdot}_{AP(\mathbb{Z})}\right), \, f \mapsto
(\langle f, T_{k\alpha}\psi \rangle )_k.$$ The map $S$ is a kind of {\it analysis map.}  Sufficiency is obtained by looking at the continuity of the formal adjoint of $S.$ When $\psi$ is a continuous function in the Wiener amalgam space, the restriction to $AP(\mathbb{Z})$  of this  adjoint is $$a=(a_k)_k \mapsto \sum_{k\in {\mathbb Z}} a_kT_{\alpha k}\psi,$$ which takes values in $AP({\mathbb R})$ (Theorem \ref{thm:fromsequencetofunction}).
\par\medskip

In section 3 we consider Gabor systems $(T_{\alpha k}M_{\beta \ell})_{k,\ell}$ with $\psi \in W_0\cup (\widehat{W_0}\cap L^1(\mathbb{R}))$ and prove the estimate
$$
\sum_{\ell\in {\mathbb Z}}\norm{\big(\langle f, T_{\alpha k}M_{\beta\ell}\psi\rangle\big)_{k\in {\mathbb Z}}}_{AP}^2 \leq C \norm{f}_{AP}^2
$$ for every continuous and almost periodic function $f.$ This is to say that the  {\it analysis} map $$(AP(\mathbb{R}), ||\cdot ||_{AP})\to \ell^2(AP(\mathbb{Z})), f\to (\langle f, T_{\alpha k}M_{\beta\ell}\psi\rangle\big)_{k, \ell}$$ is continuous. To this end, we first prove the continuity of the formal adjoint map, thus  giving a description of what can be considered as the {\it synthesis} map in this context. This is the content of Theorems \ref{thm:besselW} and \ref{thm:besselfourierW}.
\par\medskip

In section 4 we obtain generalized frames for the non separable space $AP_2({\mathbb R}).$ In particular, we prove that $AP_2({\mathbb R})$ can be viewed as a direct orthogonal
(continuous) sum of separable Hilbert spaces for which we define
frames in the usual sense. The frames are described in terms of a given Gabor frame  $(T_{\alpha k}M_{\beta \ell})_{k,\ell}$ in $L^2(\mathbb{R}).$ Due to the non separability of $AP_2({\mathbb R})$ the range of the corresponding analysis operator is not longer $\ell^2(\mathbb{Z})$ but the (completion of) the space of square summable $AP(\mathbb{Z})$-valued sequences, or, alternatively, $L^2(E)$ where $E=[0,\frac{2\pi}{\alpha})\times \mathbb{Z}$ is endowed with the counting measure. This corresponds to the concept of {\sl generalized frame} in the sense
of G. Kaiser \cite{Kais94}, S.T Ali - J.P. Antoine - J.P. Gazeau
\cite{AliAntoGaze93}, D. Han - J.P. Gabardo \cite{GabaHan03}. This  will permit to connect to Kim-Ron results
about almost periodic Gabor frames extending them to the whole space
$AP_2({\mathbb R})$.
\par\medskip
Through the paper, $\alpha$ and $\beta$ are fixed positive constant.

\section{Almost periodicity and translates}

Let us recall that the Wiener space
$W$ is the space of functions $\psi\in L^\infty(\mathbb R)$ such
that
$$
\norm{\psi}_{W}:= \sum_{k\in {\mathbb Z}}\mbox{ess\ sup}_{x\in
[0,1]}\left|\psi(x-k)\right| < \infty.
$$
Its closed subspace formed by continuous functions is denoted by
$W_0$.

The classical periodization Lemma asserts, for fixed $\alpha>0$, the boundedness of the map
$\psi\in W \longrightarrow g = \sum_{k\in {\mathbb
Z}}T_{k\alpha}\psi\in L^\infty(\mathbb R)$ (see e.g. \cite{Groc01},
Lemma 6.1.2). We start by a generalization of this Lemma to the
almost periodic case, when $\psi\in W_0$.

\begin{theorem}\label{thm:fromsequencetofunction}
\
\begin{itemize}
 \item[(a)] For every $a = \left(a_k\right)_{k\in {\mathbb Z}}\in
AP({\mathbb Z})$ and $\psi\in W_0$ we have that $$g :=
\sum_{k\in {\mathbb Z}}a_k T_{k\alpha}\psi \in AP({\mathbb R}).$$
\item[(b)] The bilinear map $$B:AP({\mathbb Z})\times W_0\to (AP({\mathbb R}),\norm{\cdot}_{AP}),\ \
(a,\psi)\mapsto \sum_{k\in {\mathbb Z}}a_k T_{k\alpha}\psi,$$ is continuous.
\end{itemize}
\end{theorem}
\begin{proof} (a) We first consider the case $\alpha = 1.$ We fix $\epsilon > 0$ and
find $L\in {\mathbb N}$ such that
$$
\sum_{|k| > L}\sup_{x\in [0,1]}\left|\psi(x-k)\right| <
\frac{\epsilon}{\norm{a}_{\infty}}.
$$ Now choose $N\in {\mathbb N}$ such that among any $N$ consecutive integers
there exists an integer $p$ with the property
$$
\left|a_{k+p} - a_k\right| < \frac{\epsilon}{(2L+1)\norm{\psi}_{\infty}}\ \
\forall k\in {\mathbb Z}.
$$ Then, for every $t\in [\ell, \ell +1],$
$$
\begin{array}{ll}
\begin{displaystyle}\left|g(t) - \sum_{|k|\leq
L}a_{k+\ell}T_{k+\ell}\psi(t)\right| \end{displaystyle}& \begin{displaystyle}=
\left|\sum_{|k|> L}a_{k+\ell}T_{k+\ell}\psi(t)\right|\end{displaystyle}\\ & \\ &
\begin{displaystyle}\leq \norm{a}_{\infty}\sum_{|k| > L}\sup_{x\in
[0,1]}\left|\psi(x-k)\right| < \epsilon\end{displaystyle}
\end{array}
$$ and also
$$
\left|g(t+p) - \sum_{|k|\leq L}a_{k+\ell+p}T_{k+\ell+p}\psi(t+p)\right| <
\epsilon
$$ for every $p\in {\mathbb Z}.$ Then every interval of length $N$ of the real
line contain at least one $3\epsilon$-period of $g$ since, for $t\in [\ell, \ell
+1],$
$$
\left|g(t+p) - g(t)\right| \leq 2\epsilon + \sum_{|k|\leq L}\left|a_{k+\ell+p} -
a_{k+\ell}\right|\cdot \left|T_{k+\ell}\psi(t)\right| < 3\epsilon.
$$
\par\medskip
In the general case, we consider $\varphi(t) = \psi(t\alpha)$ and observe that
$$
T_{k}\varphi(\frac{t}{\alpha}) = \psi(t-k\alpha) = T_{k\alpha}\psi(t)
$$ and
$$
g(t) = \big(\sum_{k\in {\mathbb Z}}a_{k}T_k\varphi\big)(\frac{t}{\alpha})
$$ is almost periodic.
\par\medskip
(b) According to \cite[5.1.4]{jarchow} it suffices to show that $B$ is separately continuous. To this end we first fix $a\in AP({\mathbb Z}).$ For every compact set $K\subset {\mathbb R}$ there is $C > 0$ such that
$$
\sup_{t\in K}\left|\sum_{k\in {\mathbb Z}}a_k T_{k\alpha}\psi(t)\right| \leq C
\norm{\psi}_{W_0}.
$$ That is, the map $B(a,\cdot):W_0\to C({\mathbb R})$ is continuous. Then, an
application of the closed graph theorem gives the continuity of $$B(a,\cdot):W_0\to (AP({\mathbb R}),\norm{\cdot}_{\infty}).$$ Since $\norm{f}_{AP}\leq \norm{f}_{\infty}$ for every $f\in AP({\mathbb R})$ we conclude that also $B(a,\cdot):W_0\to (AP({\mathbb R}),\norm{\cdot}_{AP})$ is continuous.
\par\medskip
We now fix $\psi\in W_0$ and prove the continuity of $$B(\cdot,\psi):AP({\mathbb Z})\to (AP({\mathbb R}),\norm{\cdot}_{AP}).$$ We first assume that $\psi$ vanishes outside
$[-\alpha M, \alpha M]$ and apply \cite[6.2.2]{Groc01} to get a positive constant $C,$ independent of $\psi,$ such that
$$
\norm{\sum_{k} b_k T_{\alpha k}\psi}_{2}^2 \leq C^2 \norm{\psi}_W^2 \cdot
\sum_{k} \left|b_k\right|^2
$$ for every finite sequence $\left(b_k\right)_k.$ We now consider $a = \left(a_k\right)_{k\in {\mathbb Z}}\in
AP({\mathbb Z})$ and estimate
$$
\begin{array}{ll}
&\begin{displaystyle}\norm{\sum_{k=-N}^N a_k T_{\alpha k}\psi}_{2} =
\left(\int_{-\alpha(M +  N)}^{\alpha(M + N)}\left|\sum_{k=-N}^N a_kT_{\alpha
k}\psi(t)\right|^2\ dt\right)^{\frac{1}{2}}\end{displaystyle} \\ & \\ &
\begin{displaystyle}\geq \left(\int_{-\alpha(M + N)}^{\alpha(M +
N)}\left|B(a,\psi)(t)\right|^2\ dt\right)^{\frac{1}{2}} - \left(\int_{-\alpha(M
+ N)}^{\alpha(M + N)}\left|\sum_{N < |k| < N+2M} a_kT_{\alpha
k}\psi(t)\right|^2\ dt\right)^{\frac{1}{2}}.\end{displaystyle}
\end{array}
$$ Since
$$
\int_{-\infty}^{\infty}\left|\sum_{N < |k| < N+2M} a_kT_{\alpha
k}\psi(t)\right|^2\ dt \leq C^2 \norm{\psi}_W^2\cdot \sum_{N < |k| < N+2M}
\left|a_k\right|^2
$$ and
$$
\frac{1}{2M}\sum_{N < |k| < N+2M} \left|a_k\right|^2
$$ is uniformly bounded, it turns out that
$$
\lim_{N\to \infty}\frac{1}{2N+1}\int_{-\alpha(M + N)}^{\alpha(M +
N)}\left|B(a,\psi)(t)\right|^2\ dt
$$ is less than or equal to
$$
C^2  \norm{\psi}_W^2\cdot \lim_{N\to
\infty}\frac{1}{2N+1}\sum_{k=-N}^N\left|a_k\right|^2 = C^2\norm{\psi}_W^2\cdot
\norm{a}_{AP}^2.
$$ Finally
$$
\norm{B(a,\psi)}_{AP}^2 = \lim_{T\to
\infty}\frac{1}{2T}\int_{-T}^T\left|B(a,\psi)(t)\right|^2\ dt \leq
\frac{C^2}{\alpha}\norm{\psi}_W^2\cdot \norm{a}_{AP}^2
$$ where $C$ is a constant which does not depend on $\psi$ or $a.$ In the general case that $\psi$ is not necessarily compactly supported we consider a sequence $\psi_n\in W_0$ of compactly supported functions converging to $\psi$ in $W_0.$ Since the map
$$
B(a,\cdot):W_0\to (AP({\mathbb R}),\norm{\cdot}_{AP})
$$ is continuous we get
$$
\norm{B(a,\psi)}^2_{AP} = \lim_{n\to
\infty}\norm{B(a,\psi_n}^2_{AP}\leq
\frac{C^2}{\alpha}\norm{\psi}_W^2\cdot \norm{a}_{AP}^2.
$$
\end{proof}

As mentioned in Section 1, we denote by $e_{\lambda}$ the
AP-function $e_{\lambda}(t) = e^{i\lambda t}$ and by
$\tilde{e}_{\lambda}$ the AP-sequence $\tilde{e}_{\lambda}(k) =
e^{i\lambda k},$ $k\in {\mathbb Z}$, and observe that
$\tilde{e}_{\lambda + 2\pi} = \tilde{e}_{\lambda}.$

\begin{proposition}\label{prop:fourier}{\rm
For every $a = \left(a_k\right)_{k\in {\mathbb Z}}\in
AP({\mathbb Z})$ and $\psi\in W_0,$ the $\mu$-Fourier coefficient of
$\begin{displaystyle}\sum_k a_kT_{k\alpha}\psi\end{displaystyle}$ is
$$
c_\mu = \alpha^{-1}\widehat{\psi}(\mu) \left(a, \tilde{e}_{\mu
\alpha}\right)_{AP({\mathbb Z})}.$$
}
\end{proposition}
\begin{proof}
(I) We assume at first that $\psi$ is continuous and
supported on $[-q\alpha, q\alpha].$ Then, for every $k\in {\mathbb Z},$
 $$
 \sup T_{k\alpha}\psi \subset [(-q+k)\alpha, (q+k)\alpha].
 $$ Consequently, for any $N > q,$
 $$
 \begin{array}{ll}
 \begin{displaystyle} \int_{-N\alpha}^{N\alpha}\left(\sum_k a_kT_{k\alpha}\psi\right)(t)e^{-i\mu t}\
dt \end{displaystyle} & \begin{displaystyle}=
\int_{-N\alpha}^{N\alpha}\left(\sum_{|k|\leq N-q}a_kT_{k\alpha}\psi\right)(t)e^{-i\mu t}\
dt\end{displaystyle}\\ & \\ & \begin{displaystyle}= \sum_{|k|\leq N-q}a_k
\int_{-\infty}^{\infty}\left(T_{k\alpha}
\psi\right)(t)e^{-i\mu t}\ dt\end{displaystyle} \\ & \\ & =
\begin{displaystyle}\widehat{\psi}(\mu )\sum_{|k|\leq N-q}a_k e^{-i \mu k\alpha}.
\end{displaystyle}
 \end{array}
 $$ Since
 $$
 c_\mu = \lim_{N\to \infty}\frac{1}{2N\alpha}\int_{-N\alpha}^{N\alpha}\left(\sum_k
a_kT_{k\alpha}\psi\right)(t) e^{-i\mu t}\ dt,
 $$ the conclusion follows.

(II) We consider now the general case $\psi\in W_0$.
According to Theorem \ref{thm:fromsequencetofunction} (b), for every $a = \left(a_k\right)_{k\in {\mathbb Z}}\in AP({\mathbb Z})$ and $\mu\in {\mathbb R},$
the map
$$
\psi\in W_0\mapsto \mu\mbox{-Fourier coefficient of}\  \sum_k a_kT_{k\alpha}\psi
$$ is continuous. Since $W_0\subset L^1({\mathbb R})$ with continuous inclusion,
also
$$
\psi\in W_0\mapsto \alpha^{-1}\widehat{\psi}(\mu )
$$ is continuous. Since the continuous functions with compact support are a
dense subspace of $W_0,$ an application of part (I) gives the
conclusion.
\end{proof}

The following observation will play an important role in what follows.

\begin{lemma}\label{cor:1}{\rm Let $\psi\in W \cup \widehat{W_0}$} and
$\gamma > 0$ be given. Then
$$
\sup_{\lambda}\sum_k\left|\widehat{\psi}(\lambda + \gamma k)\right|^2 < \infty.
$$

\end{lemma}
\begin{proof} The case $\psi\in \widehat{W_0}$ is well-known and follows from \cite[6.1.2]{Groc01} and the fact that $W$ is a Banach algebra with respect to pointwise multiplication. For $\psi\in W$ it suffices to use that $|\widehat{\psi}|^2\in W_0$ by \cite[2.8]{Fournier_Stewart}, and to apply again \cite[6.1.2]{Groc01}.\end{proof}


\par\medskip
We will denote by $\left<\cdot, \cdot\right>$ the usual duality involving $L^p$ spaces,
$$
\left<f, g\right> = \int_{-\infty}^{+\infty}f(t)\overline{g(t)}\ dt.
$$
\par\medskip
\begin{lemma}\label{lem:periodizationexp}{\rm \
\par
(a) For every $f\in AP({\mathbb R})$ and $\psi\in L^1({\mathbb R})$
the sequence
$$
\left(\left<f, T_{k\alpha}\psi\right>\right)_{k\in {\mathbb Z}}
$$ is almost periodic and the map
$$
S:\left(AP({\mathbb R}), \norm{\cdot}_{\infty}\right)\to \ \left(AP({\mathbb Z}), \norm{\cdot}_{\infty}\right),\ f\mapsto \left(\left<f, T_{k\alpha}\psi\right>\right)_k,
$$ is continuous.
\par
(b) For every $\lambda\in {\mathbb R},$
$$
\left(\left<e_{\lambda}, T_{k\alpha}\psi\right>\right)_k =
\overline{\widehat{\psi}(\lambda)}\cdot\tilde{e}_{\lambda \alpha}\in
AP({\mathbb Z}).
$$
\par
(c) Let $f$ be trigonometric polynomial. Then
$$
 \left(<f, T_{k\alpha}\psi>\right)_k = \sum_{\lambda \in [0,\frac{2\pi}{\alpha})}\left(\sum_p
\widehat{f}(\lambda + \frac{2\pi}{\alpha} p)\cdot\overline{\widehat{\psi}(\lambda +
\frac{2\pi}{\alpha} p)}\right)\tilde{e}_{\lambda \alpha}.
 $$
 }
\end{lemma}
\begin{proof}
(a) $\left<f, T_{k\alpha}\psi\right> = g(k)$ where $g$ is the continuous and almost periodic function $g(t) = \left(f\ast \check{\psi}\right)(t \alpha).$ Since
\begin{equation}\label{*}
\left|\left<f, T_{k\alpha}\psi\right>\right| \leq \norm{f}_\infty\cdot\norm{\psi}_1
\end{equation}

then $S$ is continuous.
\par
(b) Follows from
$$
\left<e_{\lambda}, T_{k\alpha}\psi\right> = e^{i\lambda k
\alpha}\overline{\widehat{\psi}(\lambda)}
$$ while (c) is a direct consequence of (b). As $a_\lambda\ne 0$ only for finitely many $\lambda$,
the sum in (c) is finite.
\end{proof}

Actually from \eqref{*} and the fact that $\norm{\cdot}_{AP(\mathbb Z)}\le \norm{\cdot}_\infty$, we have continuity of the sesquilinear map
$$
S:\left(AP({\mathbb R}), \norm{\cdot}_{\infty}\right)\times L^1(\mathbb R) \to \ \left(AP({\mathbb Z}), \norm{\cdot}_{AP(\mathbb Z)}\right),\ f\mapsto \left(\left<f, T_{k\alpha}\psi\right>\right)_k.
$$

\par\medskip
Next we analyze under which conditions on $\psi$ the previous map $S$ is continuous in the case that $AP({\mathbb R})$ is endowed with the almost periodic norm $\norm{\cdot}_{AP}.$
\par\medskip

\begin{proposition}\label{prop:only_trans}{\rm Given $\psi \in L^1(\mathbb{R}),$ the map $$S:\left(AP(\mathbb{R}),\norm{\cdot}_{AP}\right)\to \left( AP(\mathbb{Z}), \norm{\cdot}_{AP(\mathbb Z)} \right),\, f \to
(\langle f, T_{k\alpha}\psi \rangle )_k,$$ is continuous if and only if
$$\sup_{\lambda} \sum_{k}|\widehat{\psi}(\lambda+2 \pi
\frac{k}{\alpha})|^2<\infty.$$

In particular, this is the case when $\psi \in W$ or $\psi \in L^1(\mathbb{R})\cap \widehat{W}.$
}
\end{proposition}

\begin{proof} Let us assume that $S$ is continuous. For each $\lambda$ and each $a\in
\ell^2$ with finite support we consider the trigonometric polynomial $$f(t) = \sum_k a_k
e^{i(\lambda+2\pi \frac{k}{\alpha})t}.$$ Since $S(e^{i\mu t}) = \overline{\widehat{\psi}(\mu)}\tilde{e}_{\mu \alpha}$ then $$S(f) = \left(\sum_k a_k \overline{\widehat{\psi}}(\lambda+2\pi \frac{k}{\alpha})\right)\tilde{e}_{\lambda \alpha}$$
and $$\left|\sum_k a_k \overline{\widehat{\psi}}(\lambda+2\pi \frac{k}{\alpha})\right|\leq \norm{S(f)} \leq
\norm{S}\cdot\norm{f}_{AP} = \norm{S}\cdot \norm{a}_{\ell^2},$$ which implies that $$\sup_\lambda
\sum_{k}|\widehat{\psi}(\lambda+2 \pi \frac{k}{\alpha})|^2 \leq \norm{S}^2.$$

To prove the converse, consider the map $T:AP(\mathbb{Z})\to AP_2({\mathbb R})$
defined by assigning the Fourier coefficient
\begin{equation}\label{eq:op T}
\left(T(a), e_\lambda \right)_{AP_2}=
\widehat{\psi}(\lambda)\left(a, \tilde{e}_{\lambda
\alpha}\right)_{AP(\mathbb{Z})} \end{equation}
of $T(a)\in AP_2({\mathbb R})$, where $a\in AP(\mathbb Z)$. There is a countable set $\{\lambda_j\}_j$ contained in $[0,
\frac{2\pi}{\alpha})$ such that $\left(a, \tilde{e}_{\lambda
\alpha}\right)_{AP(\mathbb{Z})}\neq 0$ implies $\lambda = \lambda_j + 2\pi\frac{k}{\alpha}$ for some $j\in {\mathbb N}$ and $k\in {\mathbb Z}.$ Moreover, in this case,
$$
\left(a, \tilde{e}_{\lambda
\alpha}\right)_{AP(\mathbb{Z})} = \left(a, \tilde{e}_{\lambda_j
\alpha}\right)_{AP(\mathbb{Z})}.
$$ Consequently
$$
\begin{array}{ll}
\begin{displaystyle}\|T(a)\|_{AP}^2 \end{displaystyle} & \begin{displaystyle} = \sum_\lambda
|\widehat{\psi}(\lambda)\left(a, \tilde{e}_{\lambda
\alpha}\right)|^2 \end{displaystyle}\\ & \\ & \begin{displaystyle} = \sum_j |\left(a, \tilde{e}_{\lambda_j
\alpha}\right)|^2\left(\sum_k |\widehat{\psi}(\lambda_j + 2 \pi \frac{k}{\alpha}|^2\right)\end{displaystyle}\\ & \\ & \begin{displaystyle} \leq C \sum_j |\left(a, \tilde{e}_{\lambda_j
\alpha}\right)|^2\leq C\norm{a}^2_{AP(\mathbb{Z})},\end{displaystyle}\end{array}$$ which shows that $T$ is well defined and continuous. To finish we only need to check that $S$ is the restriction of $T^\ast$ to $AP({\mathbb R}).$ To this end we first observe that
$$
\left(a, T^\ast(e_\lambda)\right)_{AP_2} = \left(T(a), e_\lambda \right)_{AP_2} = \widehat{\psi}(\lambda)\left(a, \tilde{e}_{\lambda
\alpha}\right)_{AP(\mathbb{Z})},
$$ from where it follows
$$
T^\ast(e_\lambda) = S(e_\lambda).
$$ In  particular, $S$ and $T^\ast$ coincide on the trigonometric polynomials, which are a dense subspace of $\left(AP({\mathbb R}), \norm{\cdot}_\infty\right).$ Since the map $$S:\left(AP({\mathbb R}), \norm{\cdot}_\infty\right)\to AP({\mathbb Z})$$ is continuous, it easily follows that $T^\ast(f) = S(f)$ for every continuous and almost periodic function $f\in AP({\mathbb R}).$
\end{proof}

\begin{remark}\label{rem:transposed}{\rm For every $\psi\in W_0$ the maps
$$
T_\psi:AP({\mathbb Z})\to AP({\mathbb R}),\ a\mapsto \sum_{k\in {\mathbb Z}}a_k T_{\alpha k}\psi,
$$ and
$$
S_\psi:AP({\mathbb R})\to AP({\mathbb Z}),\ f\mapsto \left(\langle f, T_{\alpha k}\psi\rangle\right)_{k\in {\mathbb Z}},
$$ are, up to a constant, adjoint to each other.
\par\medskip
In fact, for $f = e_\lambda\in AP({\mathbb R})$ and $b = \tilde{e}_{\mu\alpha}$ we have, by Lemma \ref{lem:periodizationexp},
$$
\left(S_\psi f, b\right)_{AP({\mathbb Z})} = \overline{\widehat{\psi}(\lambda)}\left(\tilde{e}_{\lambda\alpha}, \tilde{e}_{\mu\alpha}\right)_{AP({\mathbb Z})}.
$$ Also, Proposition \ref{prop:fourier} gives
$$
\left(f, T_\psi b\right)_{AP} = \alpha^{-1}\overline{\widehat{\psi}(\lambda)}\left(\tilde{e}_{\lambda\alpha}, \tilde{e}_{\mu\alpha}\right)_{AP({\mathbb Z})}.
$$ Consequently
$$
S_\psi f = \alpha T_\psi^\ast f
$$ for every trigonometric polynomial $f,$ from where the conclusion follows.
 }
\end{remark}

\begin{theorem}\label{thm:orthogonalsystem}{\rm Let $\psi\in L^1({\mathbb R})$ satisfies
\begin{equation}\label{eq:cota}
C = \sup_{\lambda}\sum_p\left|\widehat{\psi}(\lambda +
\frac{2\pi}{\alpha}p)\right|^2 < \infty.
\end{equation} Then
\par
(a) There is an orthogonal system $\left(h_\lambda\right)_{\lambda\in [0, \frac{2\pi}{\alpha})}$ in $AP_2({\mathbb R})$ such that
$$
\norm{\left(<f, T_{k\alpha}\psi>\right)_{k\in {\mathbb Z}}}^2_{AP} = \sum_{\lambda\in[0, \frac{2\pi}{\alpha})}\left|\left(f, h_{\lambda}\right)_{AP_2}\right|^2
$$ for every $f\in AP({\mathbb R}).$
\par
(b) If moreover $\psi\in W_0 \cup \widehat{W_0}$ then each $h_\lambda\in AP({\mathbb R}).$
 }
\end{theorem}
\begin{proof} (a) Let $f$ be a trigonometric polynomial. Then, from Lemma \ref{lem:periodizationexp},
$$
 \left(<f, T_{k\alpha}\psi>\right)_k = \sum_{\lambda \in [0,\frac{2\pi}{\alpha})}\left(\sum_p
\widehat{f}(\lambda + \frac{2\pi}{\alpha} p)\cdot\overline{\widehat{\psi}(\lambda +
\frac{2\pi}{\alpha} p)}\right)\tilde{e}_{\lambda \alpha}
 $$ and
 $$
\norm{\left(<f, T_{k\alpha}\psi>\right)_k}_{AP}^2 = \sum_{\lambda \in
[0,\frac{2\pi}{\alpha})}\left|\sum_p \widehat{f}(\lambda + \frac{2\pi}{\alpha}
p)\cdot\overline{\widehat{\psi}(\lambda + \frac{2\pi}{\alpha}
p)}\right|^2.
$$ Now, fix $ \lambda \in [0, \frac{2\pi}{\alpha})$ and consider
$c^{\lambda}=(c^{\lambda}_\mu)_\mu$ defined
as $$ c^{\lambda}_\mu= \left \{\begin{array}{ccc}0 & \mbox{ if  } &
\alpha(\mu-\lambda)\notin 2\pi \mathbb{Z} \\  &
\\\widehat{\psi}(\lambda + \frac{2\pi}{\alpha}p)  & \mbox{ if  }&
\mu = \lambda + \frac{2\pi}{\alpha}p,\ p\in{\mathbb Z}\end{array}\right.$$ By condition (\ref{eq:cota}), there is $h_\lambda \in AP_2({\mathbb R})$ whose $\mu$-Fourier
coefficient is precisely $c^\lambda_\mu.$ Therefore, for each
trigonometric polynomial $f$,
\begin{equation}\label{eq:pol}
\norm{\left(<f, T_{k\alpha}\psi>\right)_{k\in {\mathbb Z}}}^2_{AP} = \sum_{\lambda\in[0, \frac{2\pi}{\alpha})}\left|\left(f, h_{\lambda}\right)_{AP_2}\right|^2.
\end{equation} Moreover, the functions $\left(h_{\lambda}\right )_{\lambda\in [0, \frac{2\pi}{\alpha})}$ are uniformly bounded and form an orthogonal system in $AP_2({\mathbb R})$ since they have mutually disjoint spectra. By Bessel's inequality, the map
$$
f\mapsto \left(\sum_{\lambda\in[0, \frac{2\pi}{\alpha})}\left|\left(f, h_{\lambda}\right)_{AP_2}\right|^2\right)^{\frac{1}{2}}
$$ is a continuous seminorm in $AP_2({\mathbb R}).$ Now, Lemma \ref{lem:periodizationexp} permits to conclude that (\ref{eq:pol}) also holds for every $f\in AP({\mathbb R}).$
\par\medskip
(b) In the case that $\psi\in L^1({\mathbb R})$ and $\widehat{\psi}\in W_0$ we have that
$$
h_\lambda(t) = \sum_{p}\widehat{\psi}(\lambda +
\frac{2\pi}{\alpha}p)e^{i(\lambda + \frac{2\pi}{\alpha}p)t}
$$ is a continuous and almost periodic function since
$$
\sum_{p\in {\mathbb Z}}\left|\widehat{\psi}(\lambda + \frac{2\pi}{\alpha}p)\right| < \infty.
$$
\par\medskip
For $\psi\in W_0$ we have
$$
h_\lambda(t) = \alpha\sum_{k\in {\mathbb Z}}e^{i\lambda \alpha k}T_{\alpha k}\psi (t),
$$ that is,
$$
h_\lambda = \alpha\cdot B(\tilde{e}_{\lambda\alpha}, \psi),
$$ where $B$ is the bilinear map of Theorem \ref{thm:fromsequencetofunction}. In fact, according to Proposition \ref{prop:fourier}, the Fourier coefficients of $\alpha\cdot B(\tilde{e}_{\lambda\alpha}, \psi)$ and those of $h_\lambda$ coincide.
\end{proof}
\par\medskip
The identity from Theorem \ref{thm:orthogonalsystem} can be extended to functions $f$ in the Stepanov class when $\psi\in W.$ We first need an auxiliary result.

\begin{lemma}\label{lem:Stepanov}{\rm Let $\psi \in W$ and $f\in S^2$ be given. Then
\par
 (a) $f\ast \psi\in AP({\mathbb R}),$ and
\par
(b) $\norm{\left(\left<f, T_{\alpha k}\psi\right>\right)_{k}}_{AP({\mathbb Z})}
\leq \norm{f}_{W(L^2,L^\infty)}\norm{\psi}_W.$
}
\end{lemma}
\begin{proof} (a)
 We put $\psi_k = \psi \chi_{[k,k+1]}.$ Then, for
every $x\in {\mathbb R},$
$$
\sum_{k\in {\mathbb Z}}\int_{-\infty}^{+\infty}\left|f(x-t)\right|\cdot \left|\psi_k(t)\right|\ dt $$ is less than or equal to
$$
\sum_{k\in {\mathbb Z}}\norm{f}_{W(L^2,L^\infty)}\norm{\psi_k}_\infty =
\norm{f}_{W(L^2,L^\infty)}\norm{\psi}_W.
$$ Consequently, $f\ast \psi$ is everywhere defined and bounded and $$||f\ast \psi||_\infty \leq \norm{f}_{W(L^2,L^\infty)}\norm{\psi}_W.$$ As  $P\ast \psi$ is a trigonometric polynomial for every trigonometric polynomial $P,$ we conclude that $f\ast \psi$ can be approximated in the sup-norm by a sequence of trigonometric polynomials whenever $f\in S^2.$ This gives (a)

(b) As in Lemma \ref{lem:periodizationexp} we have that $\left(\left<f, T_{\alpha k}\psi\right>\right)_{k\in {\mathbb Z}}\in AP({\mathbb Z}).$ Moreover
$$
\norm{\left(\left<f, T_{\alpha k}\psi\right>\right)_{k}}_{AP({\mathbb Z})}
 \leq \norm{f\ast \check{\psi}}_\infty \leq \norm{f}_{W(L^2,L^\infty)}\norm{\psi}_W.
$$
\end{proof}
\par\medskip
The previous result also holds under the less restrictive condition that $\psi$ belongs to the amalgam space $W(L^2, L^1).$

\begin{corollary}\label{cor:coefficientsStepanov}{\rm Let $\psi\in W$ and
$f\in \rm S^2$ be given. Then
$$
\norm{\left(<f, T_{k\alpha}\psi>\right)_{k\in {\mathbb Z}}}^2_{AP} = \sum_{\lambda\in[0, \frac{2\pi}{\alpha})}\left|\left(f, h_{\lambda}\right)_{AP_2}\right|^2,
$$ where $(h_\lambda)$ are the functions of Theorem \ref{thm:orthogonalsystem}.
}
\end{corollary}
\begin{proof}
 It follows from Theorem \ref{thm:orthogonalsystem} and Lemma \ref{lem:Stepanov}.
\end{proof}

\section{Bessel type inequalities}
\par\medskip
We recall that $\ell^2(\rm AP({\mathbb Z}))$ consists of all vector sequences $a = \left(a^\ell\right)_{\ell\in {\mathbb Z}}$ where each $a^\ell\in \rm AP({\mathbb Z})$ and  $$\norm{a}^2:= \sum_{\ell\in{\mathbb Z}}\norm{a^\ell}^2_{\rm AP({\mathbb Z})} < \infty.
 $$
 In this section we will show that, given a  Gabor system $(T_{\alpha k}M_{\beta \ell})_{k,\ell}$ with $\psi \in W_0\cup (\widehat{W_0}\cap L^1(\mathbb{R})),$  the  {\it analysis} map $$(AP(\mathbb{R}), ||\ ||_{AP})\to \ell^2(AP(\mathbb{Z})), f\mapsto (\langle f, T_{\alpha k}M_{\beta\ell}\psi\rangle\big)_{k, \ell},$$ is continuous. This could be obtained as a consequence of general results in \cite{RonShen97} with the help of Lemma \ref{lem:mkp} below. Instead, we consider the natural candidate to be the {\it synthesis map} and prove its continuity. In our opinion, this gives additional  information which can be of independent interest.  Let us recall that given a Bessel Gabor system $(T_{\alpha k}M_{\beta \ell}\psi)_{k,\ell}$ in $L^2(\mathbb{R})$ the synthesis operator is   $$(\alpha_{k,\ell})_{k,\ell} \mapsto \sum_{k,\ell}\alpha_{k,\ell}T_{\alpha k}M_{\beta \ell}\psi$$ for each square summable sequence of scalars $(\alpha_{k,\ell})_{k,\ell}.$ Hence, it is reasonable to consider the map
 $$a \mapsto \sum_{\ell\in {\mathbb Z}}\left(\sum_{k\in {\mathbb Z}}a^\ell_k T_{\alpha k}M_{\beta\ell}\psi\right)$$ for $a \in \ell^2(AP(\mathbb{Z})),$ having in mind that, for each $\ell,$ the function
 $$\sum_{k\in {\mathbb Z}}a^\ell_k T_{\alpha k}M_{\beta\ell}\psi$$ is almost periodic when $\psi\in W_0$ by Theorem \ref{thm:fromsequencetofunction}.

\begin{theorem}\label{thm:besselW}
{\rm Let $\psi\in W_0$ be given. Then the map
$$
T:\ell^2(\rm AP({\mathbb Z}))\to AP_2({\mathbb R})$$ defined as $$ T\left((a^\ell)_{\ell\in {\mathbb Z}}\right):= \sum_{\ell\in {\mathbb Z}}\left(\sum_{k\in {\mathbb Z}}a^\ell_k T_{\alpha k}M_{\beta\ell}\psi\right),
$$ is well defined and continuous. Moreover there is a constant $C$ such that
$$
\sum_{\ell\in {\mathbb Z}}\norm{\big(\langle f, T_{\alpha k}M_{\beta\ell}\psi\rangle\big)_{k\in {\mathbb Z}}}_{AP}^2 \leq C \norm{f}_{AP}^2
$$ for every $f\in AP({\mathbb R}).$
}
\end{theorem}
\begin{proof}
 We first assume that $\psi$ vanishes outside
$[-\alpha M, \alpha M]$ and apply \cite[6.2.2]{Groc01} to get a positive constant $C,$ independent of $\psi,$ such that
$$
\norm{\sum_{\ell, k} b^\ell_k T_{\alpha k}M_{\beta\ell}\psi}_{2}^2 \leq C^2 \norm{\psi}_W^2 \cdot
\sum_{\ell, k} \left|b^\ell_k\right|^2
$$ for every finite sequence $\left(b^\ell_k\right)_{k,\ell}.$ We now consider $a^\ell = \left(a^\ell_k\right)_{k\in {\mathbb Z}}\in AP({\mathbb Z})$ such that
$$
\sum_{\ell\in {\mathbb Z}}\norm{a^\ell}_{AP}^2 < \infty
$$ and, for every finite subset $J\subset {\mathbb Z}$ we can proceed as in the proof of Theorem \ref{thm:fromsequencetofunction} (b) to obtain that
$$
\lim_{N\to \infty}\frac{1}{2N+1}\int_{-\alpha(M+N)}^{\alpha(M+N)}\left|\sum_{\ell\in J}\sum_{k\in {\mathbb Z}}a^\ell_k\left(T_{\alpha k}M_{\beta\ell}\psi\right)(t)\right|^2\ dt
$$ is less than or equal to
$$
C^2\norm{\psi}_W^2\cdot\lim_{N\to \infty}\frac{1}{2N+1}\sum_{\ell\in J}\sum_{k=-N}^N\left|a^\ell_k\right|^2 = C^2\norm{\psi}_W^2\cdot\sum_{\ell\in J}\norm{a^\ell}_{AP}^2.
$$ Hence
$$
\norm{\sum_{\ell\in J}\sum_{k\in {\mathbb Z}}a^\ell_k T_{\alpha k}M_{\beta\ell}\psi}_{AP}^2 \leq \frac{C^2}{\alpha}\norm{\psi}_W^2\cdot\sum_{\ell\in J}\norm{a^\ell}_{AP}^2.
$$ As a consequence,
$$
\left\{\sum_{\ell\in J}\left(\sum_{k\in {\mathbb Z}}a^\ell_k T_{\alpha k}M_{\beta\ell}\psi\right):\ J\subset {\mathbb Z}\ \mbox{finite}\right\}
$$ is a Cauchy net in $AP({\mathbb R})$ and there exists
$$
T\left((a^\ell)_{\ell\in {\mathbb Z}}\right):= \sum_{\ell\in {\mathbb Z}}\left(\sum_{k\in {\mathbb Z}}a^\ell_k T_{\alpha k}M_{\beta\ell}\psi\right)\in AP_2({\mathbb R})
$$ with the property that
$$
T:\ell^2(AP({\mathbb Z}))\to AP_2({\mathbb R})
$$ is a continuous map. Once again, we can proceed as in Theorem \ref{thm:fromsequencetofunction} to check that the previous statement holds for every $\psi\in W_0.$ \par\medskip To obtain the estimate $$
\sum_{\ell\in {\mathbb Z}}\norm{\big(\langle f, T_{\alpha k}M_{\beta\ell}\psi\rangle\big)_{k\in {\mathbb Z}}}_{AP}^2 \leq C \norm{f}_{AP}^2
$$ it is enough to show that 
$$
T^\ast(f) = \big(\left(\left<f, T_{k\alpha}M_{\ell\beta}\psi\right>\right)_k\big)_\ell\in\ell^2(AP({\mathbb Z}))
$$ for every $f\in AP({\mathbb R}).$ To this end we denote, for each $\ell^\prime\in {\mathbb Z},$
$$
j_{\ell^\prime}:AP({\mathbb Z})\to \ell^2(AP({\mathbb Z}))
$$ the inclusion $j_{\ell^\prime}(a_k)_k):=(a_k^\ell)_{k,\ell}$ where $a_k^\ell = a_k$ for $\ell = \ell^\prime$ and $a_k^\ell = 0$ otherwise, and we define
$$
T_{\ell^\prime}(a):= \sum_{k\in {\mathbb Z}}a^{\ell^\prime}_k T_{\alpha k}M_{\beta\ell^\prime}\psi.
$$ From Remark \ref{rem:transposed},
$$
T_{\ell}^\ast(f) = j_{\ell}\left(\left(\left<f, T_{k\alpha}M_{\ell\beta}\psi\right>\right)_k\right).
$$ Since
$$
T^\ast(f) = \sum_{\ell\in {\mathbb Z}}T_{\ell}^\ast(f),
$$ the conclusion follows.
\end{proof}

Next, we consider the case $\psi\in L^1({\mathbb R}) \cap \widehat{W_0}.$ To deal with this case let us consider the following heuristic argument: first of all every element of $AP_2({\mathbb R})$ is completely determined by its Fourier coefficients. Given $a\in AP(\mathbb{Z})$ and $\psi \in W_0$ the Fourier coefficients of $$\sum_{k\in {\mathbb Z}}a^\ell_k T_{\alpha k}M_{\beta\ell}\psi$$ are $$\widehat{\psi}(\lambda - \ell\beta)\cdot \left(a^\ell, \tilde{e}_{\lambda
\alpha}\right)_{AP({\mathbb Z})}. $$ Therefore the Fourier coefficients of $$\sum_{\ell\in {\mathbb Z}}\left(\sum_{k\in {\mathbb Z}}a^\ell_k T_{\alpha k}M_{\beta\ell}\psi\right)$$ should be $$
\sum_{\ell \in {\mathbb
Z}}\widehat{\psi}(\lambda - \ell\beta)\cdot \left(a^\ell, \tilde{e}_{\lambda
\alpha}\right)_{AP({\mathbb Z})}.
$$ This  argument works when $\psi\in L^1({\mathbb R}) \cap \widehat{W_0}.$

\begin{theorem}\label{thm:besselfourierW}{\rm Let $\psi\in L^1({\mathbb R}) \cap \widehat{W_0}$ be given. Then the map $$T:\ell^2(\rm AP({\mathbb Z}))\to AP_2({\mathbb R}),$$ defined for $a =
\left(a^\ell\right)_{\ell \in {\mathbb Z}}$ by
\begin{equation}\label{eq:froml2ap}
\left(Ta, e_{\lambda}\right)_{AP_2} = \sum_{\ell \in {\mathbb
Z}}\widehat{\psi}(\lambda - \ell\beta)\cdot \left(a^\ell, \tilde{e}_{\lambda
\alpha}\right)_{AP({\mathbb Z})}
\end{equation} is well defined and continuous. Moreover, there is a constant $C$ such that
$$
\sum_{\ell\in {\mathbb Z}}\norm{\big(\langle f, T_{\alpha k}M_{\beta\ell}\psi\rangle\big)_{k\in {\mathbb Z}}}_{AP}^2 \leq C \norm{f}_{AP}^2
$$ for every $f\in AP({\mathbb R}).$
}
\end{theorem}
\begin{proof}  We first observe that
$$
\left\{\lambda\in [0, \frac{2\pi}{\alpha}):\
\left(a^\ell, \tilde{e}_{\lambda \alpha}\right) \neq 0\ \mbox{for some}\ \ell\in
{\mathbb Z}\right\}
$$ is a countable set, let say $\left\{\lambda_j\right\}_{j=1}^\infty.$ Moreover
$$
\sum_{j=1}^\infty \left|\left(a^\ell, \tilde{e}_{\lambda_j \alpha}\right)\right|^2
= \norm{a^\ell}_{AP({\mathbb Z})}^2.
$$ On the other hand, if $\left(a^\ell, \tilde{e}_{\lambda \alpha}\right) \neq 0\
\mbox{for some}\ \ell\in {\mathbb Z}$ then
$$
\lambda = \lambda_j + \frac{2\pi}{\alpha}p
$$ for some $j\in {\mathbb N}$ and $p\in {\mathbb Z}.$ We put
$$
b^j_{\ell,p} = \widehat{\psi}(\lambda_j + \frac{2\pi}{\alpha}p - \ell\beta).
$$ According to \cite[6.1.2]{Groc01} we have
$$\sup_{j,\ell}\sum_p\left|b^j_{\ell,p}\right| < \infty\ \ \mbox{and}\ \
\sup_{j,p}\sum_\ell\left|b^j_{\ell,p}\right| < \infty.
$$ Schur lemma (see \cite[6.2.1]{Groc01}) implies that, for every $j,$ the matrix
$\left(b^j_{\ell,p}\right)_{\ell,p}$ defines a bounded operator on
$\ell^2({\mathbb Z})$ and the norm of that operator is bounded by a constant $C$
independent of $j.$ In particular, the sum in (\ref{eq:froml2ap}) is absolutely
convergent and
$$
\left(\sum_\ell b^j_{\ell,p} \left(a^\ell, \tilde{e}_{\lambda_j
\alpha}\right)\right)_p \in \ell^2({\mathbb Z})
$$ has norm less than or equal to
$$
C \left(\sum_\ell \left|\left(a^\ell, \tilde{e}_{\lambda_j
\alpha}\right)\right|^2\right)^{\frac{1}{2}}.
$$ Finally
$$
\begin{array}{ll}
\begin{displaystyle}\sum_{\lambda}\left|\left(Ta, e_{\lambda}\right)_{{\rm
AP}_2}\right|^2 \end{displaystyle}& \begin{displaystyle}=
\sum_j\sum_p\left|\sum_\ell b^j_{\ell,p} \left(a^\ell, \tilde{e}_{\lambda_j
\alpha}\right)\right|^2\end{displaystyle} \\ & \\ & \begin{displaystyle}\leq C^2
\sum_j\sum_\ell \left|\left(a^\ell, \tilde{e}_{\lambda_j
\alpha}\right)\right|^2\end{displaystyle} \\ & \\ & \begin{displaystyle}= C^2
\sum_\ell \norm{a^\ell}_{AP({\mathbb Z})}^2 = C^2\norm{a}^2 < \infty.\end{displaystyle}
\end{array}
$$ This shows that $\left(\left(Ta, e_{\lambda}\right)_{AP_2}\right)$ are
the Fourier coefficients of an element $Ta \in AP_2({\mathbb R})$ and the map $T$ is
well-defined and continuous. Consequently
$$
T^\ast:AP_2({\mathbb R})\to \ell^2(B^2({\mathbb Z}))
$$ is a continuous map. Here $B^2({\mathbb Z})$ denotes the completion of the pre-Hilbert space $AP({\mathbb Z}).$ As in Theorem \ref{thm:besselW}, the proof will be complete if we show that
\begin{equation}\label{eq:adjoint}
T^\ast(f) = \big(\left(\left<f, T_{k\alpha}M_{\ell\beta}\psi\right>\right)_k\big)_\ell\in\ell^2(AP({\mathbb Z}))\end{equation} for every $f\in AP({\mathbb R}).$
\par\medskip
First, from Lemma \ref{lem:periodizationexp} we get
$$
\left(\left<e_{\lambda}, T_{k\alpha}M_{\ell\beta}\psi\right>\right)_k =
\overline{\widehat{\psi}(\lambda - \ell\beta)}\cdot\tilde{e}_{\lambda \alpha}\in
AP({\mathbb Z})
$$ and
$$
\widehat{\psi}(\lambda - \ell\beta)\cdot \left(a^\ell, \tilde{e}_{\lambda
\alpha}\right)
$$ is the inner product in $AP({\mathbb Z})$ of $a^\ell$ and $\left(\left<e_{\lambda},
T_{k\alpha}M_{\ell\beta}\psi\right>\right)_k.$ Moreover,
$$
\sum_\ell \norm{\left(\left<e_{\lambda},
T_{k\alpha}M_{\ell\beta}\psi\right>\right)_k}_{AP({\mathbb Z})}^2 =
\sum_\ell\left|\widehat{\psi}(\lambda - \ell\beta)\right|^2 < \infty,
$$ which shows that
$$
T^\ast(e_\lambda) = \left(\left<e_{\lambda}, T_{k\alpha}M_{\ell\beta}\psi\right>\right)_{k,\ell}\in
\ell^2(AP({\mathbb Z})).
$$ This proves (\ref{eq:adjoint}) in the case that $f$ is a trigonometric
polynomial. For every $\ell\in {\mathbb Z}$ we denote by $\pi_\ell$ the projection onto the $\ell$-th coordinate $\pi_\ell:\ell^2(B^2({\mathbb Z}))\to B^2({\mathbb Z}).$ Since for every $\ell\in {\mathbb
Z},$ both operators
$$
(AP({\mathbb R}),\norm{\cdot}_{\infty})\to AP({\mathbb Z}), f\mapsto
\left(\left<f, T_{k\alpha}M_{\ell\beta}\psi\right>\right)_k,
$$ and
$$
(AP({\mathbb R}),\norm{\cdot}_{\infty})\to B^2({\mathbb Z}), f\mapsto \pi_\ell\left(T^\ast
f\right),
$$ are continuous and coincide on the trigonometric polynomials we finally conclude that (\ref{eq:adjoint}) holds for every $f\in AP({\mathbb R})$ and the proof is complete.
\end{proof}

\section{Generalized frames for almost periodic functions}

Although the space $AP_2({\mathbb R})$ is not separable, in this
section we prove that it can be viewed as a direct orthogonal
(continuous) sum of separable Hilbert spaces for which we define
frames in the usual sense.

The following lemmas will be needed for Theorem
\ref{thm:KimRon-revisited} which is the main result of this section.

\begin{lemma}\label{lem:shiftsWiener}{\rm Let $\psi\in W$ be given, then  the map
 $${\mathbb R}\xrightarrow{M} W_0,\  M(\lambda) = M_{-\lambda}\psi$$ is continuous. If in addition $\psi$ is continuous, the map $$
{\mathbb R}\xrightarrow{T} W_0,\ T(\lambda) = T_{-\lambda}\psi$$  is also
continuous.}

\end{lemma}
\begin{proof}
We first prove that $M$ is continuous. In fact, we denote by $\psi_k$ the
restriction of $\psi$ to the interval $[k, k+1].$ Fix $\epsilon > 0$ and choose
$N\in {\mathbb N}$ such that
$$
\sum_{|k| > N}\norm{\psi_k}_{\infty} < \frac{\epsilon}{4}.
$$ Now take $\delta > 0$ with the property that $\left|\lambda -
\lambda_0\right| < \delta$ implies $$\left|e^{-i\lambda t} - e^{-i\lambda_0
t}\right| < \frac{\epsilon}{2\norm{\psi}_W} $$ whenever $|t| \leq N+1.$
Consequently
$$
\norm{M_{-\lambda}\psi - M_{-\lambda_0}\psi}_W \leq \frac{\epsilon}{2} +
\sum_{|k|\leq N}\mbox{ess\ sup}_{t\in [k,k+1]}\left|e^{-i\lambda t} - e^{-i\lambda_0
t}\right||\psi(t)| \leq \epsilon
$$ Since $W_0$ is translation invariant it suffices to check the continuity of
$T$ at $\lambda = 0.$ Fix $\epsilon > 0$ and choose $N\in {\mathbb N}$ such that
$$
\sum_{|k| > N}\sup_{t\in [k,k+1]}\left|\psi(\lambda + t) - \psi(t)\right| <
\frac{\epsilon}{2}
$$ for every $\lambda$ with $|\lambda| < 1.$ Now, since $\psi$ is uniformly
continuous, we can take $0 < \delta < 1$ with the property that
$\left|\lambda\right| < \delta$ implies $$\sum_{|k| \leq N}\sup_{t\in
[k,k+1]}\left|\psi(\lambda + t) - \psi(t)\right| < \frac{\epsilon}{2}.$$  Hence
$\norm{T_{-\lambda}\psi - \psi}_W < \epsilon$ and $T$ is continuous.
\end{proof}

\begin{lemma}\label{lem:mkp}
{\rm Let $\psi \in W \cup \widehat{W_0}$ and $k,p\in {\mathbb
Z}.$ Then
$$
m_{k,p}(\lambda):=\sum_{\ell\in {\mathbb Z}}\widehat{\psi}\left(\lambda +
\frac{2\pi}{\alpha}k-\ell \beta\right)\cdot
\overline{\widehat{\psi}}\left(\lambda + \frac{2\pi}{\alpha}p-\ell \beta\right)
$$ is a continuous function of $\lambda.$
}
\end{lemma}
\begin{proof} Let $X$ denote one of the Banach spaces $W$ or $\widehat{W_0},$ to which $\psi$ belongs. We consider the maps
$$
X \xrightarrow{U} \ell^2 \times \ell^2,\ \ \ell^2 \times \ell^2\xrightarrow{V}
\ell^1,\ \ {\mathbb R}\xrightarrow{M} X
$$ defined by
$$
U(\varphi) = \left(\left(\widehat{\varphi}\left(\frac{2\pi}{\alpha}k-\ell
\beta\right)\right)_{\ell},
\left(\widehat{\varphi}\left(\frac{2\pi}{\alpha}p-\ell
\beta\right)\right)_{\ell}\right),
$$ and
$$
V(a,b) = \left(a_{\ell}\cdot \overline{b}_{\ell}\right)_{\ell},\ \ M(\lambda) =
M_{-\lambda}\psi.
$$ Then $U$ and $M$ are continuous by Lemma \ref{cor:1} and Lemma
\ref{lem:shiftsWiener}, while the continuity of $V$ is obvious. Finally, the map
$$
{\mathbb R}\xrightarrow{\Theta}\ell^1,\ \ \Theta = V\circ U\circ M
$$ is continuous and the conclusion follows.
\end{proof}

Let $ {\mathcal G}(\psi, \alpha, \beta) = \left(T_{\alpha k}M_{\beta
\ell}\psi\right)_{k,\ell\in {\mathbb Z}}$ be a Gabor system where $\psi\in W \cup \left(L^1({\mathbb R})\cap \widehat{W}\right)$. According to Theorem \ref{thm:orthogonalsystem}, for every $\ell\in {\mathbb Z}$ there is an orthogonal system $\left(h_{\lambda,\ell}\right)_{\lambda\in [0, \frac{2\pi}{\alpha})}$ in $AP_2({\mathbb R})$ such that
$$
\norm{\left(<f, T_{k\alpha}M_{\beta\ell}\psi>\right)_{k\in {\mathbb Z}}}^2_{AP} = \sum_{\lambda\in[0, \frac{2\pi}{\alpha})}\left|\left(f, h_{\lambda,\ell}\right)_{AP_2}\right|^2
$$ for every $f\in AP({\mathbb R}).$

\begin{theorem}\label{thm:KimRon-revisited}
{\rm Let $\psi \in W$ or $\psi\in L^1 \cap \widehat{W}$
and assume that
$ {\mathcal G}(\psi, \alpha, \beta)$ is a Gabor frame. Then
\vskip.2cm
(a) ${\mathcal G}(\psi, \alpha, \beta)$ is an
AP-frame for $AP_2({\mathbb R})$, i.e. there exist $A,B>0$ such that
\begin{equation}\label{eq:APframe}
 A \norm{f}^2_{AP} \leq \sum_{\ell\in{\mathbb Z}}\sum_{\lambda\in[0, \frac{2\pi}{\alpha})}\left|\left(f, h_{\lambda,\ell}\right)_{AP_2}\right|^2\leq B \norm{f}^2_{AP}.
\end{equation}
for every $f\in AP_2({\mathbb R})$.

(b) In the case $\psi \in W$ we have
$$
A \norm{f}^2_{AP} \leq \sum_{\ell}\norm{\left(\left<f, T_{\alpha
k}M_{\beta \ell}\psi\right>\right)_k}_{AP}^2\leq B \norm{f}^2_{AP}.
$$ for every $f\in \rm S^2.$
 }
\end{theorem}
\begin{proof} Let us first assume that $f$ is a trigonometric polynomial. According to Lemma \ref{lem:periodizationexp}, for every $\ell\in {\mathbb Z},$
 $$
 \norm{\left(<f, T_{k\alpha}M_{\ell\beta}\psi>\right)_k}_{AP}^2$$ coincides with $$\sum_{\lambda\in [0,
\frac{2\pi}{\alpha})}\left|\sum_{k}\widehat{f}(\lambda +
\frac{2\pi}{\alpha}k)\cdot \widehat{\psi}(\lambda + \frac{2\pi}{\alpha}k -
\ell\beta)\right|^2.
 $$ All the involved sums are finite, hence
 $$
 \sum_{\ell}\norm{\left(<f, T_{k\alpha}M_{\ell\beta}\psi>\right)_k}_{AP}^2$$ can be written as $$\sum_{\lambda\in [0,
\frac{2\pi}{\alpha})}
\sum_{k,p}\widehat{f}(\lambda +
\frac{2\pi}{\alpha}k)\overline{\widehat{f}}(\lambda +
\frac{2\pi}{\alpha}p)\cdot m_{k,p}(\lambda),
 $$ where $m_{k,p}(\lambda)$ are as in Lemma \ref{lem:mkp}. On the other hand, as $\mathcal G(\psi, \alpha,\beta)$ is a Gabor frame, then
 $$
 {\mathcal G}(\widehat{\psi}, \beta, \alpha)
 $$ is also a Gabor frame and we can apply \cite{RonShen97} (see also \cite[6.3.4]{Groc01}) to
conclude that there are positive constants $A$ and $B$ with the property that,
for almost every $\lambda,$ the operator $M(\lambda):\ell^2\to \ell^2$ defined
on the finite sequences by
 $$
 \left<M(\lambda)a,b\right> = \sum_{k,p}m_{k,p}(\lambda)a_p\overline{b}_k
 $$ is a bounded operator and
 $$
 A\mbox{I}_{\ell^2} \leq M(\lambda)\leq B\mbox{I}_{\ell^2}.
 $$ Let $c = \left(c_{\ell}\right)_{\ell\in {\mathbb Z}}$ be a finite sequence.
Then
\begin{equation}\label{eq:M}
 A\norm{c}_2^2 \leq \left<M(\lambda)c,c\right> =
\sum_{k,p}m_{k,p}(\lambda)c_p\overline{c}_k \leq B\norm{c}_2^2
 \end{equation} almost everywhere. Since, from Lemma \ref{lem:mkp}, each $m_{k,p}(\lambda)$ is a continuous function, the inequalities (\ref{eq:M}) hold for every $\lambda\in
{\mathbb R}.$ In particular, for every trigonometric polynomial $f$
we obtain
 $$
 \begin{array}{ll}\begin{displaystyle}
 A \sum_{\lambda\in [0,
\frac{2\pi}{\alpha})}\sum_k\left|\widehat{f}(\lambda +
\frac{2\pi}{\alpha}k)\right|^2\end{displaystyle} & \begin{displaystyle}\leq \sum_{\ell}\norm{\left(<f, T_{k\alpha}M_{\ell\beta}\psi>\right)_k}_{AP}^2\end{displaystyle}\\ & \\ & \begin{displaystyle}\leq B
\sum_{\lambda\in [0,
\frac{2\pi}{\alpha})}\sum_k\left|\widehat{f}(\lambda +
\frac{2\pi}{\alpha}k)\right|^2.\end{displaystyle}\end{array}
 $$ That is,
 \begin{equation}\label{eq:apframe}
 A \norm{f}^2_{AP} \leq \sum_{\ell\in{\mathbb Z}}\sum_{\lambda\in[0, \frac{2\pi}{\alpha})}\left|\left(f, h_{\lambda,\ell}\right)_{AP_2}\right|^2\leq B \norm{f}^2_{AP}.
 \end{equation} Since, for each $\ell\in {\mathbb Z},$
$\left(h_{\lambda,\ell}\right)_{\lambda}$ is a uniformly bounded orthogonal system and the
trigonometric polynomials are dense in $AP_2({\mathbb R})$ it turns out that each
$\varphi_{\ell}(f):=\left(\left(f, h_{\lambda,\ell}\right)_{AP_2}\right)_{\lambda}$
defines a continuous map from $AP_2({\mathbb R})$ into $\ell^2([0,
\frac{2\pi}{\alpha}))$   and
$$
\sum_{|\ell| \leq N}\norm{\varphi_\ell(f)}^2 \leq B \norm{f}^2_{AP}
$$ for every $N\in {\mathbb N}$ and $f\in AP_2({\mathbb R}).$ Consequently,
$$
|||f|||:=\left(\sum_{\ell}\norm{\varphi_\ell(f)}^2\right)^{\frac{1}{2}} = \left(\sum_{\ell\in{\mathbb Z}}\sum_{\lambda\in[0, \frac{2\pi}{\alpha})}\left|\left(f, h_{\lambda,\ell}\right)_{AP_2}\right|^2\right)^{\frac{1}{2}}
$$ defines a continuous norm on $AP_2({\mathbb R}).$ We can conclude that
(\ref{eq:apframe}) hold for every $f\in AP_2({\mathbb R}).$ This proves (a). In particular,
\begin{equation}\label{eq:apframe2}
A \norm{f}^2_{AP} \leq \sum_{\ell}\norm{\left(\left<f, T_{\alpha
k}M_{\beta \ell}\psi\right>\right)_k}_{AP}^2\leq B \norm{f}^2_{AP}.
\end{equation} for every continuous almost periodic function $f.$ Moreover, in
the case $\psi\in W_0$ we can apply Corollary
\ref{cor:coefficientsStepanov} to conclude that (\ref{eq:apframe2})
hods for every $f\in \rm S^2,$ which proves (b).
\end{proof}

\begin{remark}{\rm (a) We observe that
$$
\sum_{\ell\in{\mathbb Z}}\sum_{\lambda\in[0, \frac{2\pi}{\alpha})}\left|\left(f, h_{\lambda,\ell}\right)_{AP_2}\right|^2 = ||\left(f, h_{\lambda,\ell}\right)_{{\rm
AP}}||_{L^2(E)}^2
$$ with $E = [0, \frac{2\pi}{\alpha})\times \mathbb Z$ endowed with discrete
(i.e. counting) measure. The statement (a)  means that the set of functions $\{h_{\lambda,\ell}\}_{(\lambda,
\ell)\in E}$ is a generalized frame in the sense of Kaiser et al. (\cite{AliAntoGaze93}, \cite{GabaHan03}, \cite{Kais94}) as mentioned in the introduction.

\noindent(b) If a Gabor frame $ {\mathcal G}(\psi, \alpha, \beta)$ defines an AP-frame  \cite{KimRon09}, one can construct for each $\ell$ the  orthogonal system $\left(h_{\lambda,\ell}\right)_{\lambda\in [0, \frac{2\pi}{\alpha})}$ in $AP_2({\mathbb R})$ and the estimate (\ref{eq:APframe}) holds for each $f\in AP_2({\mathbb R}).$

\noindent(c) Under the hypothesis of Theorem \ref{thm:KimRon-revisited}, the map 
$$
AP({\mathbb R})\to \ell^2(AP({\mathbb Z})), f\mapsto \big(\left(\left<f, T_{k\alpha}M_{\ell\beta}\psi\right>\right)_k\big)_\ell
$$ can be extended to an isomorphism from $AP_2({\mathbb R})$ into $\ell^2(AP_2({\mathbb Z})),$ where $AP_2({\mathbb Z})$ is the completion of $(AP({\mathbb Z}),\norm{\cdot}_{AP({\mathbb Z})}).$
}
\end{remark}

\begin{definition}{\rm For any countable subset $M\subset\mathbb{R}$ we denote by $AP_2(M)$ the
closed subspace of $AP_2({\mathbb R})$ consisting of those $f\in AP_2({\mathbb R})$ whose spectrum is contained in $M$.
}
\end{definition}

Then $AP_2(M)$ is a separable Hilbert space and the orthogonal projection onto $AP_2(M)$ of the family of functions $\left(h_{\lambda,\ell}\right)_{(\lambda,\ell)\in E}$ is a Hilbert frame in $AP_2(M).$ We remark that as particular case of the spaces $AP_2(M)$ we recover
the subspace of $AP_2({\mathbb R})$ of functions with rational almost periodic spectrum. We analyze some special cases.
\par\medskip
(a) Assume $\Lambda \subset [0, \frac{2\pi}{\alpha})$ is a countable set and $M = \Lambda + \frac{2\pi}{\alpha}{\mathbb Z}.$ Then the orthogonal projection of $\left(h_{\lambda,\ell}\right)_{(\lambda,\ell)\in E}$ onto $AP_2(M)$ is precisely
$$
\left(h_{\lambda,\ell}\right)_{(\lambda,\ell)\in \Lambda \times {\mathbb Z}}.
$$
\par\medskip
(b) $M = \left\{\mu_j:\ j\in {\mathbb N}\right\}$ has the property that
$$
\mu_i - \mu_j \notin \frac{2\pi}{\alpha}{\mathbb Z}
$$ whenever $i\neq j.$
\par\medskip
We put $$\mu_j = \lambda_j + \frac{2\pi}{\alpha}p_j,$$ where $\lambda_j\in [0, \frac{2\pi}{\alpha})$ and $p_j\in {\mathbb Z}$ and denote by $\pi_M:AP_2({\mathbb R})\to AP_2(M)$ the orthogonal projection. Since the spectrum of $h_{\lambda,\ell}$ is contained in $\lambda + \frac{2\pi}{\alpha}{\mathbb Z},$ then $\pi_M\left(h_{\lambda,\ell}\right) = 0$ in the case that $\lambda \neq \lambda_j$ for all $j\in {\mathbb N}.$ However, for $\lambda = \lambda_j$ we have
$$
\pi_M\left(h_{\lambda_j,\ell}\right) = \widehat{\psi}(\mu_j-\ell)e_{\mu_j}.
$$ Consequently
$$
\left\{\widehat{\psi}(\mu_j-\ell)e_{\mu_j}:\ j\in {\mathbb N}, \ell\in {\mathbb Z}\right\}
$$ is a Hilbert frame in $AP_2(M).$ We observe that if $M$ is unbounded then no system of the form
$$
\left\{\widehat{\psi}(\mu_j-\ell)e_{\mu_j}:\ j\in {\mathbb N}, \ell\in F\right\},\ F\subset {\mathbb Z}\ \mbox{finite,}
$$ can be a frame in $AP_2(M).$ In fact,
$$
\sum_{\ell \in F}\left|(e_{\mu_j}, \widehat{\psi}(\mu_j-\ell)e_{\mu_j})_{AP}\right|^2
$$ is arbitrarily small if $|\mu_j|$ is large enough.

\end{document}